\documentclass{amsart}
\usepackage{amsmath,amscd,amsthm,amssymb,amsfonts,hyperref}
\usepackage[DIV10, headinclude, BCOR 10mm]{typearea}

\usepackage{color}
\newcounter{commentcounter}

\title{On the growth of Betti numbers in $p$-adic analytic towers}

\author{Nicolas Bergeron} 

\address{Institut de Math\'ematiques de Jussieu,  
Unit\'e Mixte de Recherche 7586 du CNRS, 
Universit\'e Pierre et Marie Curie, 
4, place Jussieu 75252 Paris Cedex 05, France \\}
\email{bergeron@math.jussieu.fr}
\urladdr{http://people.math.jussieu.fr/~bergeron}

\author{Peter Linnell}
\address{Department of Mathematics\\
Virginia Tech\\
Blacksburg, VA 24061-0123,USA}
\email{plinnell@math.vt.edu}
\urladdr{http://www.math.vt.edu/people/plinnell/}

\author{Wolfgang L\"uck}
\address{Mathematisches Institut der Universit\"at Bonn\\
                Endenicher Allee 60\\
                53115 Bonn, Germany}
         \email{wolfgang.lueck@him.uni-bonn.de}
          \urladdr{http://www.him.uni-bonn.de/lueck}

\author{Roman Sauer}
\address{Institute of Algebra and Geometry \\
Karlsruhe Institute of Technology\\ 76133 Karlsruhe, Germany}
\email{roman.sauer@kit.edu}
\urladdr{http://www.math.kit.edu/iag7}

 \DeclareFontFamily{OT1}{rsfs}{}
 
\DeclareFontShape{OT1}{rsfs}{n}{it}{<-> rsfs10}{}
\DeclareMathAlphabet{\mathscr}{OT1}{rsfs}{n}{it}

\newcommand{\Z}{\mathbb{Z}}

\newcommand{\projtensor}{\widehat{\otimes}}
\newcommand{\gr}{\operatorname{gr}}

\DeclareFontFamily{OT1}{rsfs}{}

\DeclareFontShape{OT1}{rsfs}{n}{it}{<-> rsfs10}{}
\DeclareMathAlphabet{\mathscr}{OT1}{rsfs}{n}{it}
\newcommand{\Q}{\mathbb{Q}}

\newcommand{\coker}{\operatorname{coker}}

\newcommand{\rank}{\operatorname{rank}}

\newcommand{\im}{\operatorname{im}}

\newtheorem{thm}{Theorem}[section] 
\newtheorem{lem}[thm]{Lemma}         
\newtheorem*{lem*}{Lemma}         
\newtheorem{prop}[thm]{Proposition}
\newtheorem*{prop*}{Proposition}

 \newtheorem{quest}[thm]{Question}

\theoremstyle{definition}
\newtheorem{definition}[thm]{Definition}   
\newtheorem{rmk}[thm]{Remark}

\numberwithin{equation}{section}

\newcommand{\N}{\mathbb N}

\newcommand{\IF}{{\mathbb F}}

\newcommand{\IN}{{\mathbb N}}

\newcommand{\IQ}{{\mathbb Q}}

\newcommand{\IZ}{{\mathbb Z}}

\newcommand{\GL}{\mathrm{GL}}

\begin{document}

\begin{abstract}
  We study the asymptotic growth of Betti numbers in tower of finite
covers and provide simple proofs of approximation results, which were
previously obtained by Calegari-Emerton~\cite{CE1,CE2}, in the generality of
arbitrary $p$-adic analytic towers of covers.
Further, we also obtain partial results about arbitrary pro-$p$ towers. 
\end{abstract}


\maketitle


\section{Introduction and statement of results}

This paper is mainly concerned with the asymptotic growth of Betti numbers in a tower 
of finite covers of a compact space $X$ 
associated to a chain of subgroups of the fundamental group of $X$ which gives 
rise to a $p$-adic analytic group. Both Betti numbers with coefficients in $\IQ$ and $\IF_p$ are considered. Especially the case of $\IF_p$-coefficients received 
a lot of attention in recent years. We only name here the 
work of Calegari-Emerton~\cite{CE1, CE2}, which is motivated by the 
$p$-adic Langlands program, and the work of Lackenby~\cite{lackenby+tau, lackenby+large} in group theory, which 
is connected to property $\tau$ and $3$-manifold theory.

\subsection{Global setup}\label{sec:global}
With the exception of section~\ref{sec:not p-adic}, we retain the following
setup throughout this paper.  Let $X$ be a connected compact CW-complex with
fundamental group $\Gamma$. Let $p$ be a prime, let $n$ be a positive integer,
and let $\phi : \Gamma \rightarrow \GL_n (\Z_p)$ be a homomorphism. The closure
of the image of $\phi$, which is denoted by $G$, is a $p$-adic analytic group
admitting an exhausting filtration by open normal subgroups:
$$G_i = \mathrm{ker} \left( G \rightarrow \GL_n (\Z / p^i \Z) \right).$$
Set $\Gamma_i = \phi^{-1} (G_i )$, and let $X_i$ be the corresponding finite
cover of $X$. Let $\overline{X}$ be the cover of $X$ corresponding to the kernel
of $\phi$ and $\overline{\Gamma} = \Gamma / \mathrm{Ker} (\phi)$; note that
$\overline{\Gamma}$ acts properly and freely on $\overline{X}$ with quotient
$\overline{\Gamma} \backslash \overline{X} = X$.
Our main concern is the growth of the Betti numbers 
\begin{align*}
	b_k(X_i)&=\dim_\IQ H_k(X_i,\IQ);\\
	b_k(X_i,\IF_p)&=\dim_{\IF_p} H_k (X_i ,
	\mathbb{F}_p),
\end{align*}
with coefficients in $\IQ$ and $\IF_p$ as functions of~$i$.

\subsection{Growth of Betti numbers in a $p$-adic analytic tower}
W.~L\"uck proved that for each integer $k$ the
sequence $b_k (X_i )/[\Gamma : \Gamma_i]$ always converges as $i\to\infty$, 
and the limit equals the $k$-th \emph{$L^2$-Betti number} $\beta_k (\overline{X} ,
\overline{\Gamma})$ of the action of $\overline{\Gamma}$ on $\overline{X}$.  In
that context we obtain the following result on the rate of convergence in terms 
of the dimension of $G$ as a $p$-adic analytic group. We refer 
to~\cite[Theorem~8.36 on p.~201]{DDMS} for equivalent characterizations of the 
dimension of~$G$. 

\begin{thm} \label{T1} Let $d= \dim (G)$. Then, for any integer $k$ and as $i$
  tends to infinity, we have:
$$b_k (X_i) = \beta_k (\overline{X} , \overline{\Gamma}) [\Gamma : \Gamma_i] + O\left([\Gamma : \Gamma_i]^{1-{1/d}} \right).$$
\end{thm}

    The novelty of Theorem~\ref{T1} is obviously the error
    term. For more general covers, this has already been studied by Sarnak and
    Xue~\cite{SarnakXue} and by Clair and Whyte~\cite{ClairWhyte} but they
    obtain much weaker results, in particular their results don't apply when $0$
    occurs in the $L^2$-spectrum of $\overline{X}$.\footnote{We should note however that in the special case of lattices in $\mathrm{SU}(2,1)$ Sarnak and Xue produce a better exponent of $\frac{7}{12}$ in the error term for the first Betti number.} \smallskip\\
    \indent Theorem~\ref{T1} generalizes in the case of trivial
    coefficients the main theorem of Calegari-Emerton~\cite{CE1} which
    deals with arithmetic locally symmetric spaces. After this paper
had been put on the ArXiv Frank Calegari 
    informed us that Theorem~\ref{T1} can be deduced from the method of \cite{CE1}. In fact both proofs rest on a
    theorem of Harris, see Theorem~\ref{thm:harris} below, but we believe that
    our method of proof is somewhat simpler. We refer to 
    Section~\ref{sec: completed homology} for more details on the relation to 
    the work of Calegari-Emerton.

\subsection{Growth of $\IF_p$-Betti numbers in a $p$-adic analytic tower}
Homological algebra over Iwasawa algebras and the theory of $p$-adic analytic groups   provide important tools to study the asymptotic growth of 
Betti numbers 
in a $p$-adic analytic tower of covers. Whilst Iwasawa algebras are hidden in 
the proof of Theorem~\ref{T1}, they are essential even in the formulation of a 
corresponding result for $\IF_p$-Betti numbers. 
The \emph{Iwasawa algebra} of~$G$ over $R=\IF_p$ or $\IZ_p$ is the
completion of the group algebra~$R[G]$:
\[R[[G]]=\varprojlim R[G/G_i].\] The Iwasawa algebra is a right and left
Noetherian domain. Further, if $G$ is torsion-free, then $R[[G]]$ does not
contain zero divisors and its non-zero elements satisfy the Ore condition,
see~\cite[\S 6]{FarkasLinnell}.  This means that the ring of fractions
$Q(R[[G]])$ is a skew field, the \emph{Ore localization} of $R[[G]]$. Hence there 
is a notion of \emph{rank}: 

\begin{definition}
    If $G$ is torsion-free, we define the 
	\emph{rank} of a left $R [[G]]$-module $M$ as 
	\[
		\rank_{R[[G]]}(M)=\dim_{Q(R [[G]])} \bigl(Q(R [[G]])\otimes_{R [[G]]} M\bigr). 
	\]
	For general $G$ we define the \emph{rank} of $M$ as 
	\[
		\rank_{R[[G]]}(M)=\frac{1}{[G:G_0]}\rank_{R[[G_0]]}(M), 
	\]
	where $G_0<G$ is any uniform, hence torsion-free, subgroup, and $M$ is regarded as an $R[[G_0]]$-module by 
	restriction. 
\end{definition}

Using the above rank, we define an analog of $L^2$-Betti numbers in characteristic~$p$. 
For a CW-complex $Y$ the cellular chain complex will always be denoted by $C_\ast(Y)$. 
It is a consequence of the proof of Theorem~\ref{T1} (see~
\eqref{eq:for both}) that, if you replace 
$\IF_p$ by $\IZ_p$ in the definition below, you obtain the $L^2$-Betti numbers of 
$\bar X$.

\begin{definition}
	The \emph{mod $p$ $L^2$-Betti numbers} of the $\bar\Gamma$-space $\bar X$ are defined as 
	\[\beta_k (\overline{X} , \overline{\Gamma} ; \mathbb{F}_p) 
= \rank_{\mathbb{F}_p [[G]]} (H_k (\IF_p[[G]]\otimes_{\IF_p[\bar\Gamma]} C_\ast(\overline{X},\IF_p))),\]
	where $\IF_p[[G]]$ is regarded as a right $\IF_p[[\bar\Gamma]]$-module via $\phi:\Gamma\to G$. 
\end{definition}

For these characteristic~$p$ analogs of $L^2$-Betti numbers there is an approximation result 
similar to Theorem~\ref{T1}:
\begin{thm} \label{T2} Let $d= \dim (G)$. Then for any integer $k$ and as $i$
  tends to infinity, we have:
$$b_k (X_i ; \mathbb{F}_p ) = \beta_k (\overline{X} , \overline{\Gamma} ; \mathbb{F}_p ) [\Gamma : \Gamma_i] 
+ O\left([\Gamma : \Gamma_i]^{1-{1/d}} \right).$$
In particular, the limit of the sequence $b_k (X_i ;
\mathbb{F}_p)/[\Gamma:\Gamma_i]$ exists and is equal to $\beta_k (\overline{X} ,
\overline{\Gamma} ; \mathbb{F}_p)$.
\end{thm}

Here again Calegari informed us that Theorem \ref{T2} can be deduced from his joint
ongoing work with Emerton on {\it completed cohomology}. In fact one key feature of their theory 
is to set up the right framework to determine the growth rate of (mod $p$) Betti numbers even 
if the corresponding (mod $p$) $L^2$-Betti number vanishes. Proving unconditional results seems difficult; we nevertheless
point out that when $X$ is $3$-dimensional the main result of \cite{CE2} implies in particular that the error term
in Theorem \ref{T2} is the best possible in general. We also note that -- in his PhD-thesis~\cite[Theorem~5.3.1]{WallLiam} -- Liam Wall had first constructed examples of $p$-adic analytic towers of covers of a finite volume hyperbolic $3$-manifold which shows that one cannot replace the error term $O([\Gamma:\Gamma_i]^{1-1/d})$ by $O([\Gamma:\Gamma_i]^{1-1/d-\epsilon})$ for some $\epsilon>0$.

We may have $\beta_k (\overline{X} , \overline{\Gamma} ) \neq \beta_k
(\overline{X} , \overline{\Gamma} ; \mathbb{F}_p )$. An example is given
in~\cite[Example 6.2]{LLS}. One 
can even construct an example with $X$ being a \emph{manifold} (see 
Section~\ref{sec: links}):

\begin{prop}\label{prop: link complement}
  There exists a link complement $X$ and a sequence of $p$-covers $X_i$ of $X$
  such that
$$\lim_{i \rightarrow +\infty} \frac{\dim H_1 (X_i , \mathbb{F}_p)}{[\Gamma: \Gamma_i]} 
\neq \lim_{i \rightarrow +\infty} \frac{\dim H_1 (X_i , \Q)}{[\Gamma: \Gamma_i]}.$$
\end{prop}

We don't know of any example with $\overline{X}$ being aspherical.

\subsection{Beyond $p$-adic analytic groups} 
\label{sub:beyond_p_adic_analytic_groups}

The following theorem about arbitrary pro-p towers is certainly known to some experts but 
we could find no proof in the literature except in degree one.

\begin{thm}\label{thm: general existence}
	Let $k$ be a field of characteristic $p>0$. 
	Let $X$ be a compact connected CW-complex with $\Gamma$ as fundamental group. 
        Let $(\Gamma_i)_{i\ge 0}$ be a residual $p$-chain. 
	We denote the finite cover of $X$ associated to $\Gamma_i$ by $X_i$. 
	Then, for any $n\ge 0$, the sequence of normalized Betti numbers with $k$-coefficients 
	\[
		\Bigl(\frac{b_n(X_i;k)}{[\Gamma:\Gamma_i]}\Bigr)_{i\ge 0}
	\]
	is monotone decreasing and converges as $i\to\infty$. 
\end{thm}

We moreover prove that the limit is an
  integer in many situations, see 
Theorem~\ref{thm:limit for residually torsionfree nilpotent coverings} and the remark following it.

\subsection{Acknowledgments} 
\label{sub:acknowledgments}

We thank the referee for a detailed and helpful report, especially for spotting  
an error in a previous version of the proof of Theorem~\ref{thm:harris}, which is 
now corrected. 

Work on this project was supported by the Leibniz Award of W.L. granted by the DFG. 
N.B. is a member of the Institut Universitaire de France.
P.L. was partially supported by a grant from the NSA. 
R.S. thanks the Mittag-Leffler institute for its hospitality during the final stage of this project and acknowledges support by grant SA 1661/3-1 of the DFG. 

\section{Proof of Theorem~\ref{T1} and~\ref{T2}}

In the sequel we treat the cases $R=\IZ_p$ and $R=\IF_p$ simultaneously.
Depending on which case, we denote by $\dim_R(M)$ either the dimension of a
$\IF_p$-vector space or the $\IZ_p$-rank of a $\IZ_p$-module, which equals the
dimension of the $\IQ_p$-vector space $\IQ_p\otimes_{\IZ_p} M$.

As in the work of Calegari-Emerton and Emerton~\cite{CE1, CE2} the following
result of M.~Harris~\cite[Theorem~1.10]{Harris} for which corrections appear 
in~\cite{Harris-correction}, is crucial. Although not explicitly stated as such,
a proof is also contained in the work of
Farkas-Linnell~\cite{FarkasLinnell}. We give a
complete proof blending ideas from both Farkas-Linnell's and Harris' papers.

\begin{thm}[Harris]\label{thm:harris} Let $R=\IZ_p$ or $R=\IF_p$. 
	Let $M$ be a finitely generated $R[[G]]$-module. Then 
	\begin{equation}\label{eq:assertion}
		\dim_{R} \bigl(R \projtensor_{R[[G_i]]} M\bigr)=\rank_{R[[G]]}(M)[G:G_i]+O([G:G_i]^{1-1/d}).
	\end{equation} 
	Here $\projtensor$ denotes the completed tensor product. 
\end{thm}

\begin{proof} Passing to a finite index subgroup of $G$ we may, and shall, assume that $G$ is uniform and torsion-free. The proof then proceeds through a sequence of reductions.

\emph{Reduction to the case of cokernels of elements in~$R[[G]]$}. 
We first
 show that it suffices to show the theorem for $R[[G]]$-modules of the form
 \begin{equation}\label{eq:one-dimensional case}
   M=\coker\bigl(R[[G]]\xrightarrow{\_\cdot a}R[[G]]\bigr), 
 \end{equation}
 with $a\in R[[G]]$. Let $N$ be an arbitrary finitely generated
 $R[[G]]$-module. Since $R[[G]]$ is Noetherian, $N$ is finitely
 presented and we can find a matrix $A \in M(r\times s, R[[G]])$ such that
 \[
 N=\coker\bigl( R[[G]]^r\xrightarrow{r_A}R[[G]]^s\bigr), 
 \]
 where $r_A$ denotes the right multiplication 
 \[
 	r_A(x_1,\ldots, x_r)=(x_1,\ldots, x_r)\cdot A
 \]
with $A$. 
       Since the Ore localization $Q(R[[G]])$ is a skew field, 
	by row and
	column reduction in $M_r(Q(R[[G]]))$ one can find invertible matrices $B\in M(r\times r, Q(R[[G]]))$ and 
	$C\in M(s\times s, Q(R[[G]]))$ such that $D:=BAC$ is a block matrix of the form 
	\begin{align}\label{eq: the matrix D}
	    D=\begin{pmatrix} 
						I_\delta & 0_{\delta, s-\delta} \\
						0_{r-\delta, \delta}         & 0_{r-\delta\times s-\delta}
				\end{pmatrix}\in M(r\times s, R[[G]]),
	\end{align}
	where $I_\delta\in M(\delta\times\delta, R[[G]])$ is the identity matrix with $\delta\in\{0,\ldots, s\}$ and the other blocks are 
	suitable zero matrices. In other words, $D$ describes the projection onto the first $\delta$ coordinates
       $(x_1, x_2, \ldots , x_r) \mapsto (x_1, x_2, \ldots, x_{\delta}, 0,0,\ldots )$. 
       Since $B, C$ are invertible, we have $$\rank_{R[[G]]}(N)=s-\delta.$$
	There are
	nonzero $b,c\in R[[G]]$ such that $bB, Cc$ are matrices over $R[[G]]$.  We have 
	\begin{equation}\label{eq: the matrix D after killing denominators}
		(bB)A(Cc)=\begin{pmatrix} 
						bc\cdot I_\delta & 0_{\delta, s-\delta} \\
						0_{r-\delta, \delta}         & 0_{\delta\times\delta}
				\end{pmatrix}.
	\end{equation}
    Let $A_i, B_i$, and $C_i$ be the mod $G_i$ reductions of $A, bB$, and $Cc$. 
Because of $\ker(r_{A_i})\subset \ker(r_{C_i}\circ r_{A_i})=\ker(r_{A_i C_i})$ one obtains 
$\dim_{R}\ker(r_{A_i})\le\dim_{R}\ker(r_{A_iC_i})$. Because of $\im(r_{B_iA_i
C_i})=\im(r_{A_iC_i}\circ r_{B_i})\subset\im(r_{A_iC_i})$ we have
$\dim_{R}\ker(r_{A_iC_i})\le\dim_{R}\ker(r_{B_iA_iC_i})$. Therefore:
$\dim_{R}\ker(r_{A_i})\le\dim_{R}\ker(r_{B_iA_iC_i})$.  Assuming the theorem is
proved for modules as in~\eqref{eq:one-dimensional case}, this implies that
	\begin{align}\label{eq: upper estimate}
		\dim_{R} (R\projtensor_{R[[G_i]]}N)=\dim_{R} \coker(r_{A_i})
										&=\dim_{R} \ker(r_{A_i})-(r-s)[G:G_i]\notag\\
										&\le\dim_{R}\ker(r_{B_iA_iC_i})-(r-s)[G:G_i]\notag\\
										&=\dim_{R}\coker(r_{B_iA_iC_i})\\
										&=(s-\delta)\cdot [G:G_i]+O([G:G_i]^{1-1/d}). \notag
	\end{align}
	
	To prove the assertion for $N$, under the assumption that it holds for
        modules as in~\eqref{eq:one-dimensional case}, it remains to show that
      
	\begin{equation}\label{eq: lower estimate}
		\dim_{R} \coker(r_{A_i}) \ge (s-\delta) [G:G_i] - O([G:G_i]^{1-1/d}). 
	\end{equation}
        Let $E \in M((r-\delta)\times r, R[[G]])$ be the matrix such that 
\[
	R[[G]]^{r-\delta}\xrightarrow{r_E}R[[G]]^r, ~~(y_1, y_2, \ldots , y_{r-\delta})\cdot E=(0,\ldots, 0, y_1, y_2, \ldots,  y_{r-\delta}).
\]
        Let $F \in M((r-\delta)\times r, R[[G]])$ be the matrix $E\cdot (bB)$. 
        The same argument as before
        leading to~\eqref{eq: upper estimate} but now applied to $F$ shows
        \begin{equation}\label{eq:second_upper estimate}
			 \dim_{R} \coker(r_{F_i})
                          \le \delta\cdot [G:G_i]+O([G:G_i]^{1-1/d}).
	\end{equation}
        We have $0=r_D\circ r_E=r_{ED}$, yielding $ED=0$ and 
\[
	FAC=E(bB)AC=E(bI_r)BAC=(bI_r)EBAC=(bI_{r-\delta})ED=0. 
\]
Note that we used $E(bI_r)=(bI_{r-\delta})E$ here. 
        As $C$ is invertible this implies $FA=0$ and $r_A\circ r_F=r_{FA}=0$. 
        In particular, $r_{A_i}\circ r_{F_i} = 0$ and hence $\im(r_{F_i}) \subseteq \ker(r_{A_i})$. So 
        $\dim_{R} \im(r_{F_i}) \le \dim_{R} \ker(r_{A_i})$. We compute
        \begin{eqnarray*}
        \lefteqn{\dim_{R} \coker(r_{A_i}) + \dim_{R} \coker(r_{F_i})}
        & & 
        \\≈‚
        & = & 
        s \cdot [G:G_i]- \dim_{R} \im(r_{A_i}) + r \cdot [G:G_i]- \dim_{R} \im(r_{F_i})
        \\
        & \ge &
        s \cdot [G:G_i] - \dim_{R} \im(r_{A_i}) + r \cdot [G:G_i]- \dim_{R} \ker(r_{A_i})
        \\
        & = & 
        s \cdot [G:G_i].
      \end{eqnarray*}
      Now~\eqref{eq: lower estimate} follows from~\eqref{eq:second_upper estimate}.
      \smallskip \\
    \emph{Reduction to the case $R=\IF_p$.} 
        To prove the statement for a finitely generated $\IZ_p[[G]]$-module $M$ 
	  we may assume that $\rank_{\IZ_p[[G]]}(M)=0$ due to the reduction to the 
	  case~\eqref{eq:one-dimensional case} and the fact that $\IZ_p[[G]]$ has no 
	zero-divisors. For the following reason we may, in addition, assume that $M$ has no $p$-torsion: 
	Let $T = \{ m \in M \big| \exists_{d(m) \in \N}~~ p^{d(m)} \cdot m =0 \}$ be
	  its $p$-torsion part. Obviously, $T$ is a $\IZ_p [[G]]$-submodule of $M$. One easily sees by
	  additivity of dimension that 
	  \[
	  \dim_{\IQ_p}(\IQ_p\projtensor_{\IZ_p[[G_i]]} M/T)=
	  \dim_{\IQ_p}(\IQ_p\projtensor_{\IZ_p[[G_i]]} M)\text{ and }
	  \rank_{\IZ_p[[G]]}(M)= \rank_{\IZ_p[[G]]}(M/T).
	  \]
	  Hence we may assume that $M$ has no $p$-torsion. We prove now that 
	  \begin{align}\label{eq: F_p dimension equals Z_p dimension}
	  \rank_{\IF_p[[G]]}(M/pM)&=\rank_{\IZ_p[[G]]}(M), 
	    \end{align}
	  hence both are zero. 
	  Since the ring $\IZ_p[[G]]$ has finite projective dimension~\cite[Section~5.1]{ardakov+brown} and every 
	  projective $\IZ_p[[G]]$-module is free~\cite[Corollary~7.5.4 on p.~127]{Wilson} and $\IZ_p[[G]]$ is Noetherian, 
	  the finitely generated $\IZ_p[[G]]$-module $M$ possesses a finite resolution by finitely generated free $\IZ_p[[G]]$-modules: 
	\[
		0\to F_n\to F_{n-1}\to\ldots \to F_0\to M\to 0.
	\]
	  Applying the functor $N\mapsto \IF_p\otimes_{\IZ}N\cong N/pN$ yields a resolution of $M/pM$ by finitely generated, free 
	  $\IF_p[[G]]$-modules since $M$ has no $p$-torsion: 
	\[
		0\to \IF_p\otimes_{\IZ}F_n\to\ldots\to  \IF_p\otimes_{\IZ}F_0\to M/pM\to 0. 
	\]
	Now equation \eqref{eq: F_p dimension equals Z_p dimension} follows since the rank functions over $\IF_p[[G]]$ and 
	$\IZ_p[[G]]$ are additive and the equation obviously holds for finitely generated free $\IZ_p[[G]]$-modules.

    Let $N=\IZ_p\projtensor_{\IZ_p[[G_i]]} M$. Because of $\rank_{\IZ_p[[G]]}(M)=0$  
    it is enough to prove $\dim_{\IZ_p} (N)=O([G:G_i]^{1-1/d})$. 
    This follows from the $\IF_p[[G]]$-case, $\rank_{\IF_p[[G]]}(M/pM)=0$, and the inequality 
    \[
    	\dim_{\IZ_p}(N)\le \dim_{\IF_p}(N/pN)=\dim_{\IF_p}\bigl(\IF_p\projtensor_{\IF_p[[G_i]]} M/pM\bigr).
    \]
	So we reduced the proof of the theorem to the case $R=\IF_p$ and henceforth assume  $R=\IF_p$. \smallskip\\
       
  \emph{Reduction to $G$ being standard.}  We finally reduce the assertion to the
  case that $G$ is standard in the sense of~\cite[\S 8.4]{DDMS}.  Being a $p$-adic analytic group, $G$ has an open subgroup $H$ which is standard with
  respect to the manifold structure induced from $G$, see~\cite[Theorem
  8.29]{DDMS}. Since $H$ is open, we have $G_i<H$ for $i$ greater than some
  $i_0$. Being standard $H$ has a preferred collection of open normal subgroups $H_i$ which satisfy: $G_{i_0 + i-1} \subset H_i \subset G_i$ ($i\geq 1$); see e.g.,
 \cite[Ex. 6 p. 168]{DDMS}. 

Recall that we may assume that $\rank_{\IF_p[[G]]}(M)=0$ due to the reduction to the 
	  case~\eqref{eq:one-dimensional case} and the fact that $\IF_p[[G]]$ has no 
	zero-divisors.

Now if the assertion holds for $H$ with respect to the $H_i$'s, then it follows that for $i>i_0$ the
  left hand side of~\eqref{eq:assertion} is bounded by a constant times 
$ \dim_R  \bigl(R \projtensor_{R[[H_{i-i_0 +1}]]} M\bigr)$ and is therefore  $O([H:H_i]^{1-1/d}) = O([G:G_i]^{1-1/d})$,
  so the assertion holds for $G$ as well. We assume from now on that $G$ is standard 
and that $G_i = \psi^{-1} (p^i \IZ_p^d)$ where $\psi$ is the global atlas of $G$. \smallskip\\

 \emph{The remaining argument.} Let $M$ be as
        in~\eqref{eq:one-dimensional case}. We may assume $a\ne 0$.  By a
        fundamental result of Lazard, the graded ring $\gr\IF_p[[G]]$ with
        respect to the filtration $(\Delta^n)_{n\ge 0}$ by powers of the
        augmentation ideal $\Delta\subset\IF_p[[G]]$ is a polynomial algebra
        $\IF_p[X_1,\ldots, X_d]$ with indeterminates
        $X_i=x_i-1+\Delta^2$~\cite[Theorem~8.7.10 on p.~160]{Wilson}, where
        $\{x_1,\ldots,x_d\}\subset G$ is a minimal generating set. Let
        $I_i\subset \IF_p[[G]]$ be the closure of the ideal generated by
        elements $\lambda(h-1)$ with $h\in G_i$. Note that $N/I_iN\cong
        \IF_p\projtensor_{\IF_p[[G_i]]} N$ for any $\IF_p[[G]]$-module $N$.
        Since $\IF_p[[G]]$ is a domain, $\rank_{\IF_p[[G]]}(M)=0$. Now for each
        integer $i \geq 1$ (if $p>2$) or $i\geq 2$ (if $p=2$), the global atlas
        $\psi$ of $G$ induces an epimorphism $G_i \to p^i \IZ_p^d / p^{i+1}
        \IZ_p^d$ with kernel $G_{i+1}$. It therefore follows that
        $[G:G_i]=Cp^{id}$ for some rational constant $C>0$ and we have to show
        that
	\begin{equation}\label{eq:decay}
		\dim_{\IF_p} (M/I_iM)=O(p^{(d-1)i}). 
	\end{equation}
	But it follows from~\cite[Lemma 7.1]{DDMS} that there exists a positive
        integer $m$ such that $\Delta^{mp^i}\subset I_i$ for all $i$. It
        therefore suffices to show~\eqref{eq:decay} with $I_i$ replaced by
        $\Delta^{mp^{i}}$. Let $s\ge 0$ be such that $a\in
        \Delta^s\backslash\Delta^{s+1}$.  Let $a_i:\IF_p[[G]]/\Delta^{mp^i}\to
        \IF_p[[G]]/\Delta^{mp^i}$ be the map induced by right multiplication
        with $a$.  We have
	\begin{align*}
			\dim_{\IF_p} (M/\Delta^{mp^i}M) &= \dim_{\IF_p}\coker(a_i)\\
								&= \dim_{\IF_p}\ker(a_i)\\
								&= \dim_{\IF_p} (\Delta^{mp^i-s}/\Delta^{mp^i}).
	\end{align*}
	The last equality follows from the fact the graded ring is a polynomial
        ring. For the same reason the last number equals the number of monomials
        in a polynomial ring with $d$ variables each of which has total degree
        in the interval $[mp^i-s, mp^i)$.  The number of monomials of degree 
       $<  k$ is $\binom{d+k-1}{d}$. Hence
	\[
			\dim_{\IF_p} (M/\Delta^{mp^i}M)=\binom{d+mp^i-1}{d}-\binom{d+mp^i-s-1}{d}. 
	\]
	As a polynomial in $p^i$, each binomial coefficient has leading term $(mp^i)^d$. Their difference is a polynomial in $p^i$ with 
	degree at most $d-1$. This implies~\eqref{eq:decay}. 
\end{proof}

\begin{proof}[Proofs of Theorems~\ref{T1} and~\ref{T2}]
	We show for both cases $R=\IZ_p$ and $R=\IF_p$ simultaneously that 
	\begin{equation}\label{eq:for both}
	b_k (X_i ; R ) = \rank_{R[[G]]}\bigl(H_k(R[[G]]\otimes_{R\bar\Gamma}C_\ast(\overline{X}))\bigr)\cdot [\Gamma : \Gamma_i] + O\left([\Gamma : \Gamma_i]^{1-{1/d}} \right).
	\end{equation}
	The CW-structure on $X$ lifts to a $\bar\Gamma$-equivariant CW-structure on $\bar X$ 
	and to $\Gamma_i\backslash  \Gamma$-equivariant 
	CW-structures on $X_i$. We may also view $\bar X$ as a $\Gamma$-space via the 
	quotient map $\Gamma\to \bar\Gamma$. 
	Let $C_\ast(\bar X)$ be the cellular chain complex of $\bar{X}$. Each chain module $C_k(\bar X)$ 
	is a finitely generated free $\IZ[\bar \Gamma]$-module. 
	The differentials in the chain complex $C_\ast=R\otimes_\IZ C_\ast(\bar X)$ are denoted by $\partial_\ast$. 
    Note that $R\otimes_{R[\Gamma_i]} C_\ast$ is isomorphic to the cellular chain complex 
    $R\otimes_\IZ C_\ast(X_i)$ as an $R[\Gamma_i\backslash \Gamma]$-chain complex. In particular, we have 
	\begin{equation}\label{eq:cellular homology on finite covers}
		H_\ast(X_i,R)\cong H_\ast(R\otimes_{R[\Gamma_i]} C_\ast). 
	\end{equation}
	We write $\hat C_\ast$ and $\hat\partial_\ast$ short for $R[[G]]\otimes_{R\bar\Gamma}C_\ast$ and its differentials.  
	We denote 
	the cycles and boundaries in the chain complexes $\hat C_\ast$ and $C_\ast$ 
	by $\hat Z_\ast$, $\hat B_\ast$ and $Z_\ast$, $B_\ast$, respectively. 
	Let $r_n\in\IN_0$ be the rank of the finitely generated free $R[[G]]$-module $\hat C_n$. 
	In each degree $n$ we have the obvious exact sequence
	\begin{equation*}
		0\to \hat Z_n\to \hat C_n\xrightarrow{\hat\partial_n} \hat C_{n-1}\to \coker(\hat\partial_n)\to 0. 
	\end{equation*}
    By additivity of $\rank_{R[[G]]}$ we obtain that 
    \begin{align}\label{eq:rank above}
    	\rank_{R[[G]]} \bigl(H_k(\hat C_\ast)\bigr)&=\rank_{R[[G]]}\bigl(\hat Z_k/\hat B_k\bigr)\notag\\
    		&= \rank_{R[[G]]}(\hat Z_k)-\rank_{R[[G]]}(\hat B_k)\notag\\
			&= r_k-r_{k-1}+\rank_{R[[G]]}\bigl(\coker(\hat\partial_k)\bigr)-\rank_{R[[G]]}(\hat B_k)\\
			&=r_k-r_{k-1}+\rank_{R[[G]]}\bigl(\coker(\hat\partial_k)\bigr)-\bigl(r_k-\rank_{R[[G]]}\bigl(\coker(\hat\partial_{k+1})\bigr)\bigr)\notag\\
			&=\rank_{R[[G]]}\bigl(\coker(\hat\partial_k)\bigr)+\rank_{R[[G]]}\bigl(\coker(\hat\partial_{k+1})\bigr)-r_{k-1}.\notag 
    \end{align}
	By~\eqref{eq:cellular homology on finite covers}, 
	a similar argument as above, and right-exactness of the tensor product, 
	we obtain that 
	\begin{align*}
		\dim_R \bigl(H_k(X_i, R)\bigr) &= \dim_R \bigl(H_k(R\otimes_{R[\Gamma_i]} C_\ast)\bigr)\\
			&= \dim_R \bigl(\coker(R\otimes_{R[\Gamma_i]}\partial_k)\bigr)+ \dim_R \bigl(\coker(R\otimes_{R[\Gamma_i]}\partial_{k+1})\bigr)-[\Gamma:\Gamma_i]\cdot r_{k-1}\\
			&= \dim_R \bigl(R\otimes_{R[\Gamma_i]}\coker(\partial_k)\bigr)+ \dim_R \bigl(R\otimes_{R[\Gamma_i]}\coker(\partial_{k+1})\bigr)-[\Gamma:\Gamma_i]\cdot r_{k-1}.
	\end{align*}
	Hence, 
	\begin{equation}\label{eq:betti by cokernels}
		\frac{b_k(X_i,R)}{[\Gamma:\Gamma_i]}=\frac{\dim_R \bigl(R\otimes_{R[\Gamma_i]}\coker(\partial_k)\bigr)}{[\Gamma:\Gamma_i]}+ \frac{\dim_R \bigl(R\otimes_{R[\Gamma_i]}\coker(\partial_{k+1})\bigr)}{[\Gamma:\Gamma_i]}-r_{k-1}. 
	\end{equation}
    The natural map 
\[
	R\otimes_{R[\Gamma_i]} R[\Gamma]\xrightarrow{\cong} R\projtensor_{R[[G_i]]}R[[G]] 
\]	
induced by $\phi:\Gamma\to G$ is a right $R[\Gamma]$-module isomorphism (recall
that we regard $R[[G]]$ as a right $R[\Gamma]$-module via $\phi$). The inverse
is obtained as follows: Since $\Gamma_i\backslash \Gamma\cong G_i\backslash G$,
there is a natural continuous homomorphism from $G$ to the invertible elements
of the $R$-algebra $R\otimes_{R[\Gamma_i]} R[\Gamma]$.  By the universal
property of the completed group algebra there is a continuous homomorphism
$R[[G]]\to R\otimes_{R[\Gamma_i]} R[\Gamma]$ which descends to the desired
inverse.  As a consequence we get isomorphisms
	\begin{align*}
		R\projtensor_{R[[G_i]]}\coker(\hat\partial_k)& \cong R\otimes_{R[[G_i]]} R[[G]]\otimes_{R[\bar\Gamma]} \coker(\partial_k) \\
			&\cong  R\otimes_{R[[G_i]]} R[[G]]\otimes_{R[\Gamma]} \coker(\partial_k) \\
			& \cong R\otimes_{R[\Gamma_i]} \coker(\partial_k).
	\end{align*}
	and, thus, 
	\begin{equation}\label{eq:betti on finite cover}
			\frac{b_k(X_i,R)}{[G:G_i]}=\frac{\dim_R \bigl(R\projtensor _{R[[G_i]]}\coker(\hat\partial_k)\bigr)}{[G:G_i]}+ \frac{\dim_R \bigl(R\projtensor_{R[[G_i]]}\coker(\hat\partial_{k+1})\bigr)}{[G:G_i]}-r_{k-1}. 
	\end{equation}
	Now~\eqref{eq:for both} follows from~\eqref{eq:rank above},
        \eqref{eq:betti on finite cover}, and Theorem~\ref{thm:harris}.  Note
        that~\eqref{eq:for both} is exactly the statement of Theorem~\ref{T2} in
        the case $R=\IF_p$. Next we explain how Theorem~\ref{T1} follows
        from~\eqref{eq:for both} when $R=\IZ_p$. Since $\IQ_p$ has
        characteristic zero, we have $b_k(X_i)=b_k(X_i,\IQ_p)=b_k(X_i,
        \IZ_p)$. Since $b_k(X_i)/[\Gamma:\Gamma_i]\to \beta_k (\overline{X} ,
        \overline{\Gamma})$ as $i\to\infty$~\cite{Luck2}, we conclude
        	\begin{equation} \label{eq:Oreside}
		 \beta_k (\overline{X} , \overline{\Gamma})=\rank_{\IZ_p[[G]]}\bigl(H_k(\IZ_p[[G]]\otimes_{\IZ_p\bar\Gamma}C_\ast)\bigr).\qedhere
	\end{equation}
\end{proof}

\section{Relation with the completed homology}\label{sec: completed homology}

Calegari and Emerton~\cite{CE,CE2} have introduced the completed homology
groups:
\[\widetilde{H}_k = \varprojlim H_k (X_i , \mathbb{Z}_p) 
\quad \mbox{ and } \quad \widetilde{H}_k (\mathbb{F}_p) = \varprojlim H_k (X_i , \mathbb{F}_p).\]
These modules carry continuous actions of $G$ and may therefore be considered as
$\mathbb{Z}_p [[G]]$-modules or $\mathbb{F}_p [[G]]$-modules, respectively. 
In this section we 
want to clarify the relation of completed cohomology to (mod~$p$) $L^2$-Betti numbers. 

\begin{prop} 
Retaining the setup in section~\ref{sec:global} we have:
\begin{equation*}
\beta_k (\overline{X} , \overline{\Gamma}; \mathbb{F}_p ) = \mathrm{rank}_{\mathbb{F}_p [[G]]} (\widetilde{H}_k (\mathbb{F}_p) ).
\end{equation*}
\end{prop}
\begin{proof} 
  Here again we may reduce to the case where $G$ is torsion-free. 
Write $C_\ast=C_\ast(\bar X;\IF_p)$. 
The claim is equivalent to 
$$\widetilde{H}_k(\IF_p)= \varprojlim H_k (\IF_p \otimes_{\IF_p [\Gamma_i ]} C_\ast )$$
and
$$H_k(\IF_p[[G]]\otimes_{\IF_p\bar\Gamma}C_\ast ) = H_k ( \varprojlim  (\IF_p \otimes_{\IF_p [\Gamma_i ]} C_\ast ) )$$
having the same $\mathbb{F}_p [[G]]$-rank. So the statement is equivalent to: 
\begin{equation*}
Q( \IF_p
[[G]]) \otimes_{\IF_p [[\Gamma]]} H_k ( \varprojlim  (\IF_p \otimes_{\IF_p [\Gamma_i ]} C_\ast ) ) 
\cong Q( \IF_p
[[G]]) \otimes_{\IF_p [[\Gamma]]} \left( \varprojlim H_k (\IF_p \otimes_{\IF_p [\Gamma_i ]} C_\ast ) \right).
\end{equation*} 
Since $\IF_p \otimes_{\IZ_p [\Gamma_i ]} C_\ast $ is a tower of chain
complexes of abelian groups satisfying the Mittag-Leffler condition, by~\cite[Theorem 3.5.8]{Weibel} there is a short exact sequence 
\[0\to 
\varprojlim\nolimits^1 H_{k+1} (\IF_p \otimes_{\IF_p [\Gamma_i ]} C_\ast ) \to
H_k (\varprojlim (\IF_p \otimes_{\IF_p [\Gamma_i ]} C_\ast )) 
\to \varprojlim H_k (\IF_p \otimes_{\IF_p [\Gamma_i ]} C_\ast ) \to 0.
\]
Moreover, since towers of finite dimensional vector spaces over a
\emph{field} satisfy the Mittag-Leffler condition,   
we conclude that
$$\varprojlim\nolimits^1 H_{k+1} (\IF_p \otimes_{\IF_p [\Gamma_i ]} C_\ast ) =0,$$
which yields the proposition. 
\end{proof}

It follows from works of Calegari and Emerton that a similar result with $\IF_p$ replaced by $\IZ_p$ holds as well, that is, 
\begin{equation}\label{eq: completed cohom}
\beta_k (\overline{X} , \overline{\Gamma}; \mathbb{Z}_p ) = \mathrm{rank}_{\mathbb{Z}_p [[G]]} (\widetilde{H}_k),
\end{equation}
and we want to indicate why. In~\cite[Theorem~3.2]{CE1} it is shown for 
arithmetic congruence covers $X_i$ of symmetric spaces that 
the so-called co-rank $r_k$ of the completed cohomology $\widetilde H^k$ satisfies 
the equality in Theorem~\ref{T1} with $\beta_k (\bar X, \bar \Gamma)$ replaced 
by $r_k$. The proof in~\cite{CE1} is a consequence of a general result of Emerton \cite[Theorem 2.1.5]{Emerton} and their Lemma 2.2. Both these results hold not only for arithmetic congruence covers but in our generality. Since the co-rank of $\widetilde H^k$ is the same as the rank of 
$\widetilde H_k$~\cite[Theorem~1.1 (3)]{CE}, this implies~\eqref{eq: completed cohom}.

It is somewhat harder to work with completed homology, see \cite{CE1,CE2}. We nevertheless want to emphasize
that the latter contains a lot more information. It should be the right framework to determine the growth rate of 
(mod $p$) Betti numbers even if the corresponding (mod $p$) $L^2$-Betti number vanishes. However it seems hard
to extract the necessary information from completed homology, the only exception that we are aware of
is in case $X$ is $3$-dimensional, see \cite{CE2}.  

\section{Approximation results for pro-$p$ towers that are not $p$-adic analytic}\label{sec:not p-adic}

The proof of Theorem~\ref{thm: general existence} relies on the
following well-known lemma (see also~\cite{Ershof-Lueck-Osin(2012)} for a proof). 

\begin{lem}\label{lem:monotonicity lemma}
	Let $k$ be a field of characteristic $p>0$. Let $\Lambda$ be a normal subgroup in a group $\Gamma$ whose index is a $p$-power. 
	Then 
	\[
		\dim_k(k\otimes_{k[\Lambda]}M)\le [\Gamma:\Lambda]\cdot \dim_k(k\otimes_{k[\Gamma]}M). 
	\]
\end{lem}
\begin{proof}
	Because of the isomorphism 
	\[
		k\otimes_{k[\Gamma]}M\cong k\otimes_{k[\Gamma/\Lambda]} (k\otimes_{k[\Lambda]} M)
	\]
	it suffices to prove the case where $\Lambda$ is trivial and $\Gamma$ is a finite $p$-group. Let 
	\[
		k[\Gamma]^m\xrightarrow{f}k[\Gamma]^n\to M\to 0
	\]
	be a presentation of $M$. Then $\bar f=k\otimes_{k[\Gamma]} f$ is a presentation of $k\otimes_{k[\Gamma]} M$. 
	Since $\dim_k(M)=\vert\Gamma\vert\cdot n-\dim_k(\im(f))$ and $\dim_k(k\otimes_{k[\Gamma]}M)=n-\dim_k(\im(\bar f))$, we have to show 
	that 
	\[
		\dim_k(\im(f))\ge \vert\Gamma\vert\cdot\dim_k(\im(\bar f)). 
	\]
	Extend a $k$-basis $\{u_1,\ldots, u_s\}$ of $\im(\bar f)$ to a $k$-basis
        $\{u_1,\ldots, u_n\}$ of $k^n=k\otimes_{k[\Gamma]}k[\Gamma]^n$.  Let
        $x_1,\ldots, x_n$ be lifts of the $u_i$ to $k[\Gamma]^n$ such that
        $\{x_1,\ldots, x_s\}\subset\im(f)$. Since $k[\Gamma]$ is a local ring
        with the augmentation ideal as the unique maximal
        ideal~\cite[Proposition~7.5.3]{Wilson}, Nakayama's lemma implies that
        $\{x_1,\ldots, x_n\}$ generates $k[\Gamma]^n$ as a
        $k[\Gamma]$-module. Since the $k$-dimension of the $k[\Gamma]$-submodule
        generated by $x_i$ is at most $\vert\Gamma\vert$ and
        $\dim_kk[\Gamma]^n=\vert\Gamma\vert\cdot n$, the $k$-dimension of the
        $k[\Gamma]$-module generated by $\{x_1,\ldots, x_i\}$ is
        $i\vert\Gamma\vert$.  Because of $\{x_1,\ldots, x_s\}\subset\im(f)$,
	\[
		\dim_k(\im(f))\ge\vert\Gamma\vert\cdot s=\vert\Gamma\vert\cdot\dim_k(\im(\bar f))
	\]
	follows. 
\end{proof}

\begin{proof}[Proof of Theorem~\ref{thm: general existence}]
  It follows from Lemma~\ref{lem:monotonicity lemma} that, for any finitely
  presented $k[\Gamma]$-module $M$, the sequence
  $(\dim_k(k\otimes_{k[\Gamma_i]}M)/[\Gamma:\Gamma_i])_{i\ge 0}$ is monotone
  decreasing.  Let $r_n$ be the number of $n$-cells in $X$. Let $\partial_\ast$
  denote the differentials in the $k[\Gamma]$-complex
  $C_\ast(\widetilde{X};k)$. Exactly as in~\eqref{eq:betti by cokernels}, one
  has
  \[ \frac{b_n(X_i;k)}{[\Gamma:\Gamma_i]}=\frac{\dim_k(k\otimes_{k[\Gamma_i]}\coker(\partial_n))}{[\Gamma:\Gamma_i]}+
  \frac{\dim_k(k\otimes_{k[\Gamma_i]}\coker(\partial_{n+1}))}{[\Gamma:\Gamma_i]}-r_{n-1},
  \]
  from which monotonicity follows. Since Betti numbers are non-negative, the
  sequence converges.
\end{proof}

In the remainder of this section we study the question how we can express the limit 
\[
	\Bigl(\frac{b_n(X_i;k)}{[\Gamma:\Gamma_i]}\Bigr)_{i\ge 0}
\]
by some algebraic expression in very specific situations. For that we recall the
notions of ordered group, the Malcev-Neumann power series ring, and the division
closure.

An \emph{ordered group} is a group with a strict total ordering of its elements
which is invariant under left and right translations. A group which has such an
ordering is called \emph{orderable}.  For example, residually torsion-free
nilpotent groups are orderable~\cite[Proposition~1.2 on p.~274]{Braids}.

If $\Gamma$ is an ordered group, then the set of formal power series
$\sum_{\gamma\in\Gamma}a_\gamma\gamma$ with coefficients $a_\gamma$ in a skew
field $k$ whose support $\{\gamma\in\Gamma~\vert~a_\gamma\ne 0\}$ is
well-ordered becomes a skew field with the obvious ring structure extending the
one of the group ring~$k[\Gamma]$~\cite[Corollary~15.10 on p.~95]{Cohn} which is
called the \emph{Malcev-Neumann power series ring}.  We denote it
by~$k((\Gamma))$\footnote{We suppress the order in the notation for reasons to
  be seen below.}.

Suppose $R$ is a subring of a skew field $K$.  Then $D(R,K)$ will denote the
\emph{division closure} of $R$ in $K$, that is the smallest skew subfield of $K$
that contains $R$. If $M_1$ and $M_2$ are the Malcev-Neumann power series rings
of $k[\Gamma]$ with respect to two different orders and $D_1$, $D_2$ the division
closures of $k[\Gamma]$ in $M_1$, $M_2$, respectively, then there is a ring
isomorphism $D_1\cong D_2$ which is the identity on $k[\Gamma]$.  This follows
from the next theorem.
As a consequence, the dimension of the $k((\Gamma))$-vector space
\[k((\Gamma))\otimes_{k[\Gamma]}M\cong k((\Gamma))\otimes_{D(k[\Gamma],
  k((\Gamma)))}\bigl(D(k[\Gamma], k((\Gamma)))\otimes_{k[\Gamma]}M\bigr)\] for a
$k[\Gamma]$-module $M$ does not depend on the choice of the order on~$\Gamma$.

\begin{thm} \label{THughes}
Let $k$ be a skew field, let $\Gamma$ be an orderable group, and let
$M_1,M_2$ be Malcev-Neumann power series rings for $k[\Gamma]$.  Set
$D_1 = D(k[\Gamma ], M_1)$ and $D_2 = D(k[\Gamma], M_2)$.  Then there
is a ring isomorphism $\theta \colon D_1 \to D_2$ such that $\theta$
is the identity on $k[\Gamma]$.
\end{thm}
\begin{proof} 
  Recall~\cite[Corollary~1.4]{Rhmetulla-Rolfson(2002)}  that, being orderable, the
  group $\Gamma$ is locally indicable, i.e., each finitely generated subgroup
  $H$ surjects onto $\mathbb{Z}$.
Let $J \lhd H$ be the kernel of this surjection and pick $t \in H$ so
that $H = \langle J,t \rangle$.  Let $N_1$ and $N_2$ denote the
subrings of $M_1$ and $M_2$ respectively consisting of power series
with supports in $J$.  Then $D(kJ,D_1) \subseteq N_1$ and $D(kJ,D_2)
\subseteq N_2$.  Clearly $\{t^n \mid n \in \mathbb{N}\}$ is linearly
independent over $N_i$ in $M_i$ for $i = 1,2$.
  Theorem~\ref{THughes} therefore follows from a slight
  generalization of Hughes' theorem~\cite[7.1 Theorem]{DHS04}; all that one
  needs to do is to modify the proof of \cite[7.2 Lemma]{DHS04}; this is done
  explicitly in~\cite[Hughes' Theorem I 6.3]{Sanchez08}.
\end{proof}

\begin{thm}\label{thm:limit for residually torsionfree nilpotent coverings}
  Let $k$ be a field of characteristic $p>0$.  Let $\Gamma$ be a finitely
  generated group, and let $\Gamma = \Gamma_0 > \Gamma_1 > \cdots$ be a descending
  sequence of normal subgroups such that $\bigcap_i\Gamma_i=\{1\}$ and each
  $\Gamma/\Gamma_i$ is torsion-free nilpotent.  Set $H_i=\Gamma_i\Gamma^{p^i}$
  (thus $\Gamma/H_i$ is a finite $p$-group).  If $X$ is a compact connected
  CW-complex with fundamental group $\Gamma$ and $X_i$ the finite cover
  corresponding to $H_i$, then\footnote{Recall that the left hand side
    of~\eqref{eq:betti special tower} does not depend on the order we choose to
    define $k((\Gamma))$.}, for every $n$,
  \begin{equation}\label{eq:betti special tower}
    \dim_{k((\Gamma))}\bigl(H_n(k((\Gamma))\otimes_{k[\Gamma]} C_\ast(\widetilde{X}, k))\bigr)=\lim_{i\to\infty}\frac{b_n(X_i;k)}{[\Gamma:H_i]}.
  \end{equation}
  In particular, the limit on the right hand side is an integer.
\end{thm}

\begin{rmk}
  A large collection of groups are known to be residually torsion-free nilpotent
  {(RTFN)}: Carl Droms has proved in his PhD thesis~\cite[Theorem~1.1 in
  Chapter~III on page~58]{Droms(1983)}  (see also~\cite{Agol}) that graph groups
  are {RTFN}. But this property is inherited by subgroups. It therefore follows
  from~\cite{HW} and~\cite{BHW} that free groups,
  surface groups, reflection groups, right-angled Artin groups and arithmetic
  hyperbolic groups defined by  quadratic forms are virtually {RTFN}. It follows from the recent breakthrough 
  of Agol \cite{Agol2012} and Wise \cite{Wise} that the fundamental group of a closed hyperbolic $3$-manifold is virtually {RTFN}.
  We finally note that
  the proof of~\cite[Corollary~2.3]{Agol} implies that direct and free products
  of RTFN groups are RTFN.
\end{rmk}

One deduces the preceding theorem from the following
Proposition~\ref{prop:special tower} by a similar, even easier, argument as used
in the deduction of Theorem~\ref{T1} from Theorem~\ref{thm:harris}. More
precisely: Similarly as in~\eqref{eq:rank above} and~\eqref{eq:betti by
  cokernels} one expresses the left and right hand side of~\eqref{eq:betti
  special tower} by the cokernels of the $n$-th and $(n+1)$-th differential of
the complexes $k((\Gamma))\otimes_{k[\Gamma]} C_\ast(\widetilde{X}, k)$ and
$C_\ast(X_i, k)$, respectively; then one applies the proposition below.

\begin{prop}\label{prop:special tower}
	Let $k$ be a field of characteristic $p>0$.  Let $\Gamma$ be a finitely generated group, and let 
	$\Gamma = \Gamma_0 > \Gamma_1 > \cdots$ be a descending sequence of normal subgroups such that 
	$\bigcap_i\Gamma_i=\{1\}$ and each $\Gamma/\Gamma_i$ is torsion-free nilpotent. 
	Set $H_i=\Gamma_i\Gamma^{p^i}$. 
	If $M$ is a  finitely presented $k[\Gamma]$-module, then 
\[
\dim_{k((\Gamma))} (k((\Gamma)) \otimes_{k[\Gamma]} M)=\lim_{i\to \infty} \frac{\dim_k(k\otimes_{k[H_i]} M)}{[\Gamma:H_i]}.\]
In particular, the limit on the right hand side is an integer. 
\end{prop}

A free group is residually torsion-free nilpotent. But even for a free group we
cannot say anything for arbitrary residual $p$-chains.  In particular, the
following question remains open.

\begin{quest}
Let $F$ be a finitely generated free group, let $k$ be a field of
characteristic $p>0$, and let $F=F_0 > F_1 > \cdots$ be a
descending sequence of normal subgroups with $F/F_i$ a finite
$p$-group for all $i$ and $\bigcap_{i\in \mathbb{N}} F_i = 1$.  Let
$M$ be a finitely presented $kG$-module.  Can $\lim_{i\to \infty}
|F/F_i|^{-1} \dim_k (k\otimes_{k[F_i]} M)$ be transcendental?
\end{quest}

In the remainder of this section we are concerned with the proof of Proposition~\ref{prop:special tower} for 
which we need the following lemma. 

\begin{lem} \label{Lnilpotent}
Let $k$ be an arbitrary skew field, let $M$ be a finitely presented
$k[\Gamma]$-module, and let
$\Gamma = \Gamma_0 > \Gamma_1 > \cdots$ be a descending sequence of normal
subgroups with $\Gamma/\Gamma_i$ torsion-free nilpotent for all $i$ with
$\bigcap_{i\in \mathbb{N}} \Gamma_i = 1$.
Let $D_i$ denote the
skew field of fractions of $k[\Gamma/\Gamma_i]$, which is an Ore localization. 
Then there exists $m \in
\mathbb{N}$ such that
\[\dim_{D_i} (D_i \otimes_{k[\Gamma]} M) = \dim_{k((\Gamma))} (k((\Gamma))\otimes_{k[\Gamma]} M)~~\text{for all $i \ge m$}.\]
\end{lem}

\begin{proof}
Choose a nonprincipal ultrafilter $\omega$ on $\mathbb{N}$ and let
$D$ denote the ultraproduct of the $D_i$ with respect to the
ultrafilter $\omega$.  Then $D$ is a skew field and $k[\Gamma]$ embeds in $D$.
Consider a nontrivial finitely generated subgroup $H$ of $\Gamma$.  Let
$n$ be the least positive integer such that $H$ is not contained in
$\Gamma_n$.  Then $H/H \cap \Gamma_n$ is a nontrivial finitely generated
torsion-free nilpotent group, so there exists $N \lhd H$ such that
$H/N$ is infinite cyclic.  Choose $t \in H \setminus N$ so that $H =
\langle N,t\rangle$.  Then $\{t^i \mid i \in \mathbb{N}\}$ is
linearly independent over $k((N))$ and it follows that $\{t^i \mid i
\in \mathbb{N}\}$ is linearly independent over $D(k[N],k((\Gamma)))$.
Next let $E_i$ denote the
skew field of fractions of $k[N/N\cap \Gamma_i]$ and form the ultraproduct
$E$ of the $E_i$ with respect to $\omega$.  Then $E$ is a skew field
contained in $D$, and $\{t^i \mid i \in \mathbb{N}\}$ is linearly
independent over $E$.  Therefore $\{t^i \mid i\in \mathbb{N}\}$ is
linearly independent over $D(k[N],D)$ and we deduce from
\cite[Hughes' Theorem I 6.3]{Sanchez08}
that there is an isomorphism $\theta \colon
D(k[\Gamma],k((\Gamma))) \to D(k[\Gamma],D)$ such that $\theta$ is the identity on $k[\Gamma]$.
Therefore $\dim_{k((\Gamma))} k((\Gamma)) \otimes_{k[\Gamma]} M = \dim_D D
\otimes_{k[\Gamma]} M$.  Also $\dim_D D \otimes_{k[\Gamma]} M = \lim_{\omega}
\dim_{D_i} D_i \otimes_{k[\Gamma]} M$ and the result follows.
\end{proof}

\begin{proof}[Proof of Proposition~\ref{prop:special tower}]
  For each $i\in \mathbb{N}$, let $D_i$ denote the skew field of fractions of
  $k[\Gamma/\Gamma_i]$.  By Lemma~\ref{Lnilpotent}, there exists $m_0\in
  \mathbb{N}$ such that $\dim_{D_m} D_m \otimes_{k[\Gamma]} M = \dim_{k((\Gamma))}
  k((\Gamma)) \otimes_{k[\Gamma]} M$ for all $m \ge m_0$.  For $i, m\ge m_0$, set
  $K_i^m = \Gamma_m\Gamma^{p^i}$.  For any $m\ge m_0$ we have
  \[\lim_{i\to \infty} |\Gamma/K_i^m|^{-1}
  \dim_k (k \otimes_{k[K_i^m]}M) = \dim_{D_m} (D_m \otimes_{k[\Gamma]}
  M)=\dim_{k((\Gamma))} k((\Gamma)) \otimes_{k[\Gamma]} M\] 
by \cite[Theorem~6.3]{LLS}.  The result follows from the monotonicity
  lemma~\ref{lem:monotonicity lemma}.
\end{proof}

\section{Link complements}\label{sec: links}
Suppose that $\Gamma \twoheadrightarrow \Z^d$. We embed $\Z^d
\hookrightarrow \Z_p^d =: G$ and consider the homomorphism $\phi : \Gamma
\rightarrow G$. The $\Z[\Z^d]$-module $H_1 (\overline{X} , \Z )$ is the
Alexander invariant of $X$. There is a natural map $\Z [\Z^d ] \rightarrow
\mathbb{F}_p [[ \Z_p^d ]]$. Since $\Z^d$ is amenable, the group ring
$\mathbb{F}_p [\Z^d ]$ is an Ore domain and it follows from~\cite{LLS} that
\begin{equation*}
  \begin{split}
    \beta_k (\overline{X} , \overline{\Gamma} ; \mathbb{F}_p) & =
    \rank_{\mathbb{F}_p [[\Z_p^d]]} H_k (\overline{X} , \mathbb{F}_p [[\Z_p^d]]) \\ & =  \mathrm{rank}_{\mathbb{F}_p [\Z^d ]}  H_k (\overline{X} , \mathbb{F}_p [\Z^d]) \\
    & = \mathrm{rank}_{\mathbb{F}_p [t_1^{\pm 1} , \ldots , t_d^{\pm1}]} H_k
    (\overline{X} , \mathbb{F}_p [t_1^{\pm 1} , \ldots , t_d^{\pm1}]).
  \end{split}
\end{equation*}
This last expression is zero if and only if the (first) Alexander polynomial
$\Delta$ of $X$ is non zero modulo $p$.\medskip\\
The above remark in particular applies when $X$ is a $3$-manifold
with boundary a union of $d$ tori. If $X$ is knot complement (in which case
$d=1$) the Alexander polynomial $\Delta$ is nonzero and its coefficients are
relatively prime. Hence
\begin{equation} \label{E1} \beta_1 (\overline{X} , \overline{\Gamma} ;
  \mathbb{F}_p) = \rank_{\mathbb{F}_p [[\Z_p^d]]} H_1 (\overline{X} ,
  \mathbb{F}_p [[\Z_p^d]])=0
\end{equation}
and (see also~\cite[Corollary 4.4]{SW2})
$$\lim_{i \rightarrow +\infty} \frac{\dim H_1 (X_i , \mathbb{F}_p)}{[\Gamma: \Gamma_i]} =0.$$
In that case one may even deduce from \eqref{E1} that $H_1 (X_i , \mathbb{F}_p )
= \mathbb{F}_p$ for all $i$, see~\cite[Lemma 5.4]{CE2}.\medskip\\
We may as well consider the case of a link complement $l=l_1 \cup
\ldots \cup l_d$ in $\mathbb{S}^3$. Recall that, in the case $d=2$, $\Delta (1)$
is equal to the linking number $\mathrm{Lk} (l_1 , l_2)$ of the two components
of the link. The same proof as in the knot case then shows that if $p$ is a
prime that does not divide $\mathrm{Lk} (l_1 , l_2)$, then
$$\lim_{i \rightarrow +\infty} \frac{\dim H_1 (X_i , \mathbb{F}_p)}{[\Gamma: \Gamma_i]} =0.$$
One may wonder if the proof may be extended to show that $H_1 (X_i ,
\mathbb{F}_p ) = \mathbb{F}_p^2$ for all $i$ as is true according
to~\cite[Theorem 5.11]{SW2}.

We conclude this note by the proof of Proposition~\ref{prop: link complement}. 

\begin{proof}[Proof of Proposition~\ref{prop: link complement}] It follows from~\cite{Hosokawa} that there exists a link $l$ with
  $2$ components such that $\Delta (t,t) = p$. Now note that $\Delta (t,t)$ is
  the Alexander polynomial associated to the abelian cover of $X$ corresponding
  to the map $\Gamma \rightarrow \Z^2 \rightarrow \Z^2 / \langle a-b \rangle
  \cong \Z$.  Since $\Delta (t,t)$ is non zero,
$$ \lim_{i \rightarrow +\infty} \frac{\dim H_1 (X_i , \Q)}{[\Gamma: \Gamma_i]} =0.$$
But since $\Delta (t,t)$ is zero modulo $p$, we have:
\[\lim_{i \rightarrow +\infty} \frac{\dim H_1 (X_i , \mathbb{F}_p)}{[\Gamma: \Gamma_i]} \neq 0.\qedhere\]
\end{proof}

Other examples  of closed finite $CW$-complexes
with a  chain $\pi_1(X) = \Gamma_0 \supseteq \Gamma_1 \supseteq \Gamma_2 \supseteq \cdots$ 
of in $\pi_1(X)$  normal subgroups of finite index
such that $\lim_{i \rightarrow +\infty} \frac{\dim H_1 (X_i , \mathbb{F}_p)}{[\Gamma: \Gamma_i]} 
\neq \lim_{i \rightarrow +\infty} \frac{\dim H_1 (X_i , \Q)}{[\Gamma: \Gamma_i]}$ holds
can be found in~\cite{Ershof-Lueck-Osin(2012)} and~\cite{Lueck-homology}.
One can additionally  arrange $X = B\Gamma$ or  $\bigcap_{i \ge 0} \Gamma_i = \{1\}$. However,
the  problem is still open to find an example with both $X = B\Gamma$ and  $\bigcap_{i \ge 0} \Gamma_i = \{1\}$.

\bibliography{bibli}

\def\cprime{$'$} \def\cprime{$'$} \def\cprime{$'$}
\begin{thebibliography}{10}

\bibitem{Agol}
I.~Agol.
\newblock Criteria for virtual fibering.
\newblock {\em J. Topol.}, 1(2):269--284, 2008.

\bibitem{Agol2012}
I.~Agol.
\newblock The virtual {H}aken conjecture.
\newblock Preprint, arXiv:1204.2702v1 [math.GR], with appendix by
  Agol-Groves-Manning, 2012.

\bibitem{ardakov+brown}
K.~Ardakov and K.~A. Brown.
\newblock Ring-theoretic properties of {I}wasawa algebras: a survey.
\newblock {\em Doc. Math.}, Extra Vol.:7--33, 2006.

\bibitem{BHW}
N.~Bergeron, F.~Haglund, and D.~T. Wise.
\newblock Hyperplane sections in arithmetic hyperbolic manifolds.
\newblock {\em Journal of the London Mathematical Society}, 83(2):431--448,
  2011.

\bibitem{CE}
F.~Calegari and M.~Emerton.
\newblock Completed cohomology of arithmetic groups.
\newblock preprint of preliminary version available at {\tt
  http://www.math.northwestern.edu/~fcale/}.

\bibitem{CE1}
F.~Calegari and M.~Emerton.
\newblock Bounds for multiplicities of unitary representations of cohomological
  type in spaces of cusp forms.
\newblock {\em Ann. of Math. (2)}, 170(3):1437--1446, 2009.

\bibitem{CE2}
F.~Calegari and M.~Emerton.
\newblock Mod-{$p$} cohomology growth in {$p$}-adic analytic towers of
  3-manifolds.
\newblock {\em Groups Geom. Dyn.}, 5(2):355--366, 2011.

\bibitem{ClairWhyte}
B.~Clair and K.~Whyte.
\newblock Growth of {B}etti numbers.
\newblock {\em Topology}, 42(5):1125--1142, 2003.

\bibitem{Cohn}
P.~M. Cohn.
\newblock {\em Free ideal rings and localization in general rings}, volume~3 of
  {\em New Mathematical Monographs}.
\newblock Cambridge University Press, Cambridge, 2006.

\bibitem{Braids}
P.~Dehornoy, I.~Dynnikov, D.~Rolfsen, and B.~Wiest.
\newblock {\em Ordering braids}, volume 148 of {\em Mathematical Surveys and
  Monographs}.
\newblock American Mathematical Society, Providence, RI, 2008.

\bibitem{DHS04}
W.~Dicks, D.~Herbera, and J.~S{\'a}nchez.
\newblock On a theorem of {I}an {H}ughes about division rings of fractions.
\newblock {\em Comm. Algebra}, 32(3):1127--1149, 2004.

\bibitem{DDMS}
J.~D. Dixon, M.~P.~F. du~Sautoy, A.~Mann, and D.~Segal.
\newblock {\em Analytic pro-{$p$} groups}, volume~61 of {\em Cambridge Studies
  in Advanced Mathematics}.
\newblock Cambridge University Press, Cambridge, second edition, 1999.

\bibitem{Droms(1983)}
C.~Droms.
\newblock {\em Graph Groups}.
\newblock PhD thesis, Syracuse University, 1983.

\bibitem{Emerton}
M.~Emerton.
\newblock On the interpolation of systems of eigenvalues attached to
  automorphic {H}ecke eigenforms.
\newblock {\em Invent. Math.}, 164(1):1--84, 2006.

\bibitem{Ershof-Lueck-Osin(2012)}
M.~Ershov, W.~L{\"u}ck, and D.~Osin.
\newblock The first $l^2$-{B}etti number and approximation in prime
  characteristics.
\newblock in preparation.

\bibitem{FarkasLinnell}
D.~R. Farkas and P.~A. Linnell.
\newblock Congruence subgroups and the {A}tiyah conjecture.
\newblock In {\em Groups, rings and algebras}, volume 420 of {\em Contemp.
  Math.}, pages 89--102. Amer. Math. Soc., Providence, RI, 2006.

\bibitem{HW}
F.~Haglund and D.~T. Wise.
\newblock Special cube complexes.
\newblock {\em Geom. Funct. Anal.}, 17(5):1551--1620, 2008.

\bibitem{Harris}
M.~Harris.
\newblock {$p$}-adic representations arising from descent on abelian varieties.
\newblock {\em Compositio Math.}, 39(2):177--245, 1979.

\bibitem{Harris-correction}
M.~Harris.
\newblock Correction to: ``{$p$}-adic representations arising from descent on
  abelian varieties'' [{C}ompositio {M}ath.\ {\bf 39} (1979), no.\ 2, 177--245;
  {MR}0546966 (80j:14035)].
\newblock {\em Compositio Math.}, 121(1):105--108, 2000.

\bibitem{Hosokawa}
F.~Hosokawa.
\newblock On {$\nabla $}-polynomials of links.
\newblock {\em Osaka Math. J.}, 10:273--282, 1958.

\bibitem{lackenby+tau}
M.~Lackenby.
\newblock Large groups, property {$(\tau)$} and the homology growth of
  subgroups.
\newblock {\em Math. Proc. Cambridge Philos. Soc.}, 146(3):625--648, 2009.

\bibitem{lackenby+large}
M.~Lackenby.
\newblock Detecting large groups.
\newblock {\em J. Algebra}, 324(10):2636--2657, 2010.

\bibitem{LLS}
P.~Linnell, W.~L{\"u}ck, and R.~Sauer.
\newblock The limit of {$\Bbb F_p$}-{B}etti numbers of a tower of finite covers
  with amenable fundamental groups.
\newblock {\em Proc. Amer. Math. Soc.}, 139(2):421--434, 2011.

\bibitem{Luck2}
W.~L{\"u}ck.
\newblock Approximating {$L\sp 2$}-invariants by their finite-dimensional
  analogues.
\newblock {\em Geom. Funct. Anal.}, 4(4):455--481, 1994.

\bibitem{Lueck-homology}
W.~L\"uck.
\newblock Approximating $l^2$-invariants and homology growth.
\newblock arXiv:1203.2827v2, 2012.

\bibitem{Rhmetulla-Rolfson(2002)}
A.~Rhemtulla and D.~Rolfsen.
\newblock Local indicability in ordered groups: braids and elementary amenable
  groups.
\newblock {\em Proc. Amer. Math. Soc.}, 130(9):2569--2577 (electronic), 2002.

\bibitem{Sanchez08}
J.~S{\'a}nchez.
\newblock {\em Localization: On Division Rings and Tilting Modules}.
\newblock PhD thesis, Universitat Aut\`onoma de Barcelona, 2008.

\bibitem{SarnakXue}
P.~Sarnak and X.~X. Xue.
\newblock Bounds for multiplicities of automorphic representations.
\newblock {\em Duke Math. J.}, 64(1):207--227, 1991.

\bibitem{SW2}
D.~S. Silver and S.~G. Williams.
\newblock Torsion numbers of augmented groups with applications to knots and
  links.
\newblock {\em Enseign. Math. (2)}, 48(3-4):317--343, 2002.

\bibitem{WallLiam}
L.~Wall.
\newblock Homology in finite index subgroups.
\newblock PhD Thesis available on http://www.maths.ox.ac.uk/~lackenby/foo.pdf,
  2009.

\bibitem{Weibel}
C.~A. Weibel.
\newblock {\em An introduction to homological algebra}, volume~38 of {\em
  Cambridge Studies in Advanced Mathematics}.
\newblock Cambridge University Press, Cambridge, 1994.

\bibitem{Wilson}
J.~S. Wilson.
\newblock {\em Profinite groups}, volume~19 of {\em London Mathematical Society
  Monographs. New Series}.
\newblock The Clarendon Press Oxford University Press, New York, 1998.

\bibitem{Wise}
D.~T. Wise.
\newblock The structure of groups with a quasiconvex hierarchy.
\newblock Preprint, 2009.

\end{thebibliography}

\bibliographystyle{abbrv}

\end{document}